\documentclass[a4paper,12pt]{amsart}
\usepackage[T1]{fontenc} 
\usepackage[all]{xy}
\usepackage[english]{babel}
\usepackage{frcursive}
\usepackage{eucal}
\usepackage{graphicx}
\usepackage{epsfig}
\usepackage{mathrsfs}
\usepackage{amssymb}
\usepackage{amsxtra}
\usepackage{enumerate}
\input cyracc.def

\newtheorem{theorem}{Theorem}[section]
\newtheorem{proposition}{ Proposition}[section]
\newtheorem{lemma}{ Lemma}[section]
\newtheorem{corollary}{Corollary}[section]

\newtheorem{definition}{Definition}[section]
\theoremstyle{remark}

\def \1{\mathbb {1}}
\def \RM{\mathbb {R}}
\def \NM{\mathbb{N}}
\def \ZM{\mathbb{Z}}
\def \CM{\mathbb{C}}

\def \QM{\mathbb{Q}}
\def \PM {\mathbb{P}}

\def \intr {{\rm int\,}}

\def \Der {{\rm Der\,}}

\def \Aut {{\rm Aut\,}}

\def \p {{\rm exp\,}}
\def \Id {{\rm Id\,}}

\def \d{\partial}
 
\def\a{\alpha}
\def\b{\beta}
\def\e{\varepsilon}  
\def\g{\gamma}

\def\l{\lambda}

\def\p{\varphi}

\def\G{\Gamma}   

\def \s{\sigma}

\def \to{\longrightarrow} 

\def \alg{\mathfrak{g}}

\def \< {{\thengle }}
\def \> {{\rangle }}
\def \( {\left( }
\def \) {\right) }

\newcommand{\Bt}{{\mathcal B}}
\newcommand{\Ct}{{\mathcal C}}

\newcommand{\Ht}{{\mathcal H}}
\newcommand{\It}{{\mathcal I}}

\newcommand{\Lt}{{\mathcal L}}
\newcommand{\Mt}{{\mathcal M}}

\newcommand{\Ot}{{\mathcal O}}

\renewcommand{\mod}{{\rm  mod\,}}
\title[DEGENERATIONS OF INVARIANT LAGRANGIAN MANIFOLDS]{DEGENERATIONS OF INVARIANT LAGRANGIAN~MANIFOLDS}
\author{  Mauricio  Garay}
\address{ Institut f\"ur Mathematik\\
FB 08 - Physik, Mathematik und Informatik\\
Johannes Gutenberg-Universit\"at Mainz\\
Staudinger Weg 9\\
55128 Mainz.}
\begin{document}
\begin{abstract} We consider a pair $(H,I)$ where $I$ is an involutive ideal of a Poisson algebra and $H \in I$. We show that if $I$ defines a $2n$-gon singularity then, under arithmetical conditions on $H$, any deformation of $H$ can "integrated" as a deformation of $(H,I)$.
\end{abstract}
\maketitle
\section*{Introduction}
In a family of Lagrangian submanifolds of a symplectic manifold, the fibres can degenerate and acquire singularities. One of the most typical degenerations is given by the {\em $2n$-gon singularity}: 
$$L_\e=\{ (q,p) \in \CM^{2n}:p_1^2+q_1^2=\e_1,\dots,p_n^2+q_n^2=\e_n \},$$
with  $\e=(\e_1,\dots,\e_n) .$ The name of this singularity comes from the fact that $L_0$ is the cone over a polytope obtained by taking the convex envelope of $2n$-points in  $\PM^{2n-1}$.

The geometry of this Lagrangian degeneration is quite obvious. For $n=1$,   the Riemann surface
$$R_\e:=\{ (q,p) \in \CM^2: q^2+p^2=\e\} $$ 
is a cylinder, provided that $\e \neq 0$. It retracts to its real part and thus has the homotopy type of a circle.  For $\e=0$, the Riemann surface degenerates into a cone which again retracts on its real part: a single point. This is the classical $A_1$ singularity~\cite{AVGL1}.

Now for $n>1$,  the complex Lagrangian variety $L_\e$ is the product
$$R_{\e_1} \times R_{\e_2} \times \dots \times R_{\e_n} ,$$ and therefore it retracts to a real $n$-dimensional torus $(S^1)^n$
provided that all $\e_i$ are non zero. When some of the $\e_i$ tend to zero, this smooth complex manifold degenerates into a singular Lagrangian variety, which also retracts on its real part:   a torus of dimension strictly lower than~$n$. Note that, like in the one-dimensional case, this real torus is 
the singular locus of our complex Lagrangian variety.
 
 Lagrangian $2n$-gon singularities are special fibres of Liouville integrable systems. In the seventies', Vey proved that inside the class of Liouville integrable systems these singularities are symplectically stable~\cite{Vey}. However, it is known since Poincar\'e that integrable systems are an exceptional class of dynamical systems:  under a small perturbation an integrable Hamiltonian motion will, as a general rule, not be integrable~\cite{Poincare_Methodes}. So it could happen that the degenerations provided by integrable systems are highly non-generic and exotic, even among Lagrangian degenerations.

Quite opposite to this point of view, Colin de Verdi\`ere asked if Lagrangian complete intersection singularities can always be expressed as special fibres of integrable systems~\cite{Colin}. As far as I know, this question is still open. However, in~\cite{moment}, I proved a variant of the conjecture, namely that the singularities observed generically in integrable systems occur generically as degenerations of Lagrangian manifolds. 

For $2n$-gon singularities, this result was, in fact, first stated by Nguyen D'uc  and Pham  but they mistakenly used Mather's preparation theorem in their proof~\cite{Pham_Legendre,Mather_stability}. The differential relations involved in the definition of a Lagrangian manifolds  do not permit to apply directly Mather-Thom theory. Using  Sevenheck-van Straten's  Lagrangian deformation complex, the method can nevertheless be adapted~\cite{lagrange,moment,VS}~(see also~\cite{Colin_San,VanStraten_Lagrangian}). 
 
 Now this result still left open the question of the genericity of Lagrangian degenerations as invariant manifolds in Hamiltonian systems.      
  In this paper, I will show that the $2n$-gon degeneration is an unavoidable Lagrangian singularity: as a family of invariant Lagrangian manifolds, they define symplectically stable degenerations over some Cantor type subset. Away from the singular fibres, we recover the standard KAM theorem, which states that most smooth invariant Lagrangian fibres are stable~\cite{Arnold_KAM,Kolmogorov_KAM,Moser_KAM}.
  
   The stability of the $2n$-gon degeneration does NOT always hold for the real $C^\infty$ case: as shown by Miranda and Vu Ngoc depending on its real form, it might happen  that the $2n$-gon degeneration is not stable~\cite{Miranda}. Nevertheless, to go from the complex analytic case to the real one is a detail: whatever real form we consider, real analytic variants for the stability of the $2n$-gon degenerations are straightforward.

 \section{Lagrangian deformations}
 Deformations of Poisson algebras is a well established subject going back to the work of Lichnerowicz~\cite{Lichnerowicz_Poisson}.
 
 Let $A $ be an algebra over  $\CM$. We say that $A$ is a {\em Poisson algebra}  if it is endowed of a linear antisymmetric biderivation
 $$A \times A \to A,\ (f,g) \mapsto \{ f,g \}  $$
 which satisfies the Jacobi identity. The tensor product of a Poisson algebra $A$ with an algebra $B$ is a Poisson algebra for the bracket:
 $$\{a_1 \otimes b_1,a_2 \otimes b_2 \}:=\{ a_1,a_2 \} \otimes (b_1 b_2). $$
 It is called the {\em central extension} of $A$ with respect to $B$. In the sequel, if $A$ has a Poisson structure, then we implictly consider $A \otimes B$ with this central-extension Poisson structure.
 
 The most standard example of such an algebra is the polynomial ring in $2n$ variables 
 $$\CM[q,p]:=\CM[q_1,\dots,q_n,p_1,\dots,p_n] $$
 with the symplectic Poisson structure
$$\{f,g \}=\sum_{i=1}^n \d_{q_i}f\d_{p_i}g-\d_{q_i}g\d_{p_i}f .$$
 In this paper, we will only be concerned with central extension of symplectic structures. 
 
 An ideal $I$ of a Poisson algebra is called {\em involutive} if:
 $$\{ I,I \} \subset I. $$
 and we consider flat deformations of involutive ideals inside Poisson algebra.

So if $f_1,\dots,f_n$ generates an involutive ideal $I$ then there exists
$c_{ij}^k \in A$ such that
$$\{ f_i,f_j\}=\sum_{k \geq 0}c_{ij}^k f_k. $$ 
In practise, our ideals are complete intersection ideals, thus flat involutive deformations are simply deformations of the functions which define the 
 ideal and which remain in involution.
 
Let now $B$ be an algebra and consider  the Poisson algebra $A \otimes B$. Let $\It$ be an involutive flat deformation of $I$ over $B$. We say that it is a {\em trivial Poisson deformation} if there exists a Poisson automorphism $\p \in \Aut(A \otimes B )$ such that 
$$\p(\It)=I  \otimes 1 .$$
We say that $I$ is {\em  rigid over $B$} if any of its deformation over $B$ is trivial.

These notions extend naturally from polynomial rings to analytic ones.
Let us denote by $\Ot_k$ or by $\CM\{ x \}$ when we want to single out the variables, the algebra of convergent power series in the variable $x=(x_1,\dots,x_k)$. 

This algebra has a natural topological structure as direct limit of Banach spaces (see~\cite{Grothendieck_EVT}). Because the algebra $\Ot_k \otimes \Ot_l $ is not isomorphic to $\Ot_{k+l}$, in the sequel, we consider topological tensor products rather than usual tensor products (see~\cite{Grothendieck_PTT}).
As both $\Ot_k \hat \otimes \Ot_l$ and $\Ot_{k+l}$ are completions of the space of polynomials, they are isomorphic.

 Similarly, if $X$ and $Y$ are compact spaces, by Stone theorem, multiplication gives an isomorphism of topological vector spaces between $C^0(X,\RM) \hat \otimes C^0(Y,\RM)$ and 
$C^0(X \times Y,\RM)$. These are the only properties of topological tensor products that we shall need and can be taken as a definition, if the reader is not comfortable with this notion.

Although deformation theory of involutive ideals in Poisson algebras can be established in great generality, only a few examples are understood. We now give the simplest non trivial case.

 \begin{theorem}[\cite{lagrange}] Let us consider the ring of analytic power series in $2n$-variables $\CM\{ q,p \}$ together with its symplectic Poisson structure.
The involutive ideal of $\CM\{\l, q,p \}$ generated by the polynomials 
$p_1q_1-\l_1,\dots,p_nq_n-\l_n $
is rigid over $\CM\{ t_1,\dots,t_k \}$ for any $k \geq 0$.
\end{theorem}
In concrete terms, if we take convergent power series
$$F_1,\dots,F_n \in  \CM\{t, \l, q,p \}$$
with $F_i(t=0,\l,q,p)=p_iq_i-\l_i$ that generate an involutive ideal then there exists a Poisson automorphism 
$$\p:\CM\{t, \l, q,p \}\to \CM\{t, \l, q,p \} $$ such that
$$\p(F_i) =p_iq_i-\l_i $$
for any $i=1,\dots,n$.

\section{Invariant Lagrangian manifolds}
Let $A$ be a topological Poisson algebra and $I$ an involutive ideal. As $I$ is involutive, any element $H \in I$ defines a derivation
$$A/I \to A/I,\ f \mapsto \{ H,f \} $$ 
and this derivation is unchanged if we add an element of $I^2$:
$$\{ H+\sum_{i=1}^k a_ib_i,f \}=\{ H,f \} +\sum_{i=1}^k a_i \{ b_i,f \}+\sum_{i=1}^k b_i\{ a_i,f \}=  \{ H,f \}\ (\mod I)  $$
for any $a_1,\dots,a_k$, $b_1,\dots,b_k \in I$. This shows that the conormal ideal $I/I^2$ is mapped to the derivations of $A/I$ via the mapping
$$I/I^2 \to \Der(A/I),\ H \mapsto \{H,-\}.  $$

Conversely given a function $H \in A$ any involutive ideal which contains $H$ defines an invariant variety for $H$. In such a case, we say that the ideal
is {\em $H$-invariant}. The simplest case is the invariant ideal generated by $H$ itself which corresponds to the conservation
of energy.

Let us consider the case where $A$ is the ring of formal power series $\CM[[t, \l,q,p]]$ in $4n$-variables endowed with $t=(t_1,\dots,t_n)$. We use the grading
where the degrees are:
$$\deg(q_i)=\deg(p_i)=1,\ \deg(\l_i)= 2,\ \deg(t_i)=0$$
together with the associated filtration.
We write $f=g+o(k)$ if $f-g$ contains only terms of degree higher than $k$. If $f=o(k)$ we say that it is of {\em order $k$.}

Note that the graduation of a ring $A$ induces a graduation for its derivations: {\em a derivation $v$ is homogeneous of degree  $k$}  if it maps
$Gr^i(A)$ to $Gr^{i+k}(A)$. Similarly a filtration of a ring induces a filtration of its derivations.
\begin{proposition}
\label{P::formal} . Let $\a_1,\dots,\a_n$ be real numbers which are linearly independent over $\QM$.
Let 
$$H_0 =\sum_{i=1}^n (\a_i+t_i) p_iq_i \in \CM[t, \l,q,p]$$ 
and let $I \subset \CM[[t, \l,q,p]]$ be the ideal generated by the
$p_iq_i-\l_i$'s then any function of the type
$$H=H_0+o(2).$$
there exists a sequence $(u_k)$ such that:
\begin{enumerate}[{\rm i)}]
\item $u_k$ is a hamiltonian derivation of order $2^{k-1}$ and degree $2^k$ ;
\item $\p(H)=H_0\ \mod (I^2 \oplus \CM[[t,\l]]).$
\end{enumerate}
\end{proposition}
\begin{proof} Write
$$H_0=H+R_k+o(2^{k+1}),\ R_k \in \CM[[t,\l,q,p]].  $$
where $R_k$ is of order $2^{k}$ and degree at most $2^{k+1}$. We prove the proposition by induction on $k$. 

As the $\a_i$'s are linearly independent, we may find a hamiltonian derivation
$$v=\sum_{i=1}^n a_i \d_{t_i}+\{F,-\}, $$
of order $2^{k-1}$ and degree at most $2^{k}$, and $S \in I^2$ such that:
$$v(H_0)=\{ F,H_0 \} +\sum_{i=1}^n a_i p_iq_i=R+S. $$
As $v$ is Hamiltonian, the automorphism $e^{-v} $ is a Poisson automorphism.
The element $e^{-v}H $ is of the form $H_0+o(2^{k+2})  \mod I^2$.
This proves the proposition.
\end{proof}
The proposition implies that any power series $H  \in \CM[[t, \l,q,p]]$ of the form
$$\sum_{i=1}^n (\a_i+t_i) p_iq_i +o(2)$$
with $ [\QM(\a_1,\dots,\a_n:\QM]=n$ admits an invariant ideal of dimension $n$ over  $\CM[[t, \l]]$.
Our main result is an analytic variant of this proposition. 

\section{Statement of the theorem}
We first define arithmetic classes.

  Denote by $(\cdot,\cdot)$ the Euclidean scalar product in $\CM^n$.  For any vector $\a \in \CM^n $, we define the sequence $\s(\a)$ by:
$$\s(\a)_k :=\min \{ |(\a,i)|: i \in \ZM^n \setminus \{ 0 \}, \| i \| \leq 2^k \} .$$
\begin{definition} The arithmetic class in $\CM^n$ associated to a real decreasing sequence $a=(a_k)$ is the set
$$\Ct(a):=\{\a \in \CM^n: \s(\a)_k  \geq  a_k\} .$$
\end{definition}
 Although the arithmetic class depends on the dimension $n$, we do not specify it in our notation.

  For a given closed subset $X \subset \CM^n$ and a given $l \in \NM \cup \{ \infty \}$, we  denote by $C^l(X)$ the complex valued Whitney $C^l$ functions on $X$~\cite{Whitney_extension}.

Fix a sequence $a=(a_k)$ and consider the associated arithmetic class $ \Ct(a) \subset \CM^n$. The algebra of continuous functions
on $\Ct(a)$ in a neighbourhood of~$\a$ is the direct limit taken over open neighbourhoods of $\a$:
$$C_{\a}^l(a):=\lim_{\to} C^l(\Ct(a) \cap U). $$ Thus as any direct limit of topological vector spaces, it  has a natural topological structure~\cite{Dieudonne_Schwartz,Grothendieck_EVT}. 

If $E,F$ are topological vector spaces, we denote by $\Lt(E,F)$ the space of bounded linear mappings endowed with the strong topology.
\begin{theorem}
\label{T::stable}
Let $a=(a_n)$ be a decreasing positive sequence such that
$$\sum_{n\geq 0} \frac{\log a_n}{2^n}>-\infty. $$  
Take $\a \in \Ct(a) \subset \CM^n$ and consider the morphism of Poisson algebras induced by the inclusion
$\CM\{ \l \} \subset C_{\a}^\infty(a)$:
$$ r:\CM\{t,\l,q,p\} \to \CM\{\l, q,p \} \hat \otimes C_{\a}^\infty(a)  .$$ 
Let  $I \subset \CM\{t,\l,q,p\} $ be the   involutive ideal generated by the
$p_iq_i+\l_i$'s and consider a holomorphic function of the type
$$H=\sum_{i=1}^n (\a_i+t_i) p_iq_i+o(2) \in \CM\{t,\l,q,p\}.$$
There exists a sequence a $1$-bounded morphism $u_\bullet$ such that
\begin{enumerate}[{\rm i)}]
\item  $u_k$ is a polynomial derivation of order $2^{k-1}$ and degree at most $2^k$~;
\item the sequence $\p_k=(e^{u_k}\dots e^{u_0}) $ converges in $\Lt( \CM\{t,\l,q,p\},\CM\{\l, q,p \} \hat \otimes C_{\a}^\infty(a) )$ to a
Poisson morphism ;
\item $\p(H)=\sum_{i=1}^n (\a_i+t_i) p_iq_i\ (\mod(r(I^2) \oplus   \CM\{\l \} \hat \otimes C_{\a}^\infty(a)).$
\end{enumerate}
Moreover if $H$ is real for some antiholomorphic involution then the $u$ can also be chosen real.
\end{theorem}
The theorem implies that, in a neighbourhood of the origin,
the hamiltonian flow of $H$ admits  a family of invariant Lagrangian manifolds parametrised by a neighbourhood of $(0,\a) \in \CM^n \times \Ct(a)$.
Regularity results show that it is sufficient to prove the theorem for the $C^0$ case~(see Appendix).

As we shall see the theorem is a consequence of the abstract KAM theorem:

\begin{theorem}[\cite{Abstract_KAM}]
 \label{T::KAM} Let $E_\bullet$ be an Arnold space, $M_\bullet $ a closed subspace of $E_\bullet$ and
 $F_\bullet,G_\bullet$ a bounded splitting of $M_\bullet $. Consider a vector subspace $\alg_\bullet \subset \Mt^1(E_\bullet)^{(2)}$ and let $a \in E_0$ be such that $\alg_\bullet$ maps $a+F_\bullet$ in $G_\bullet$.  We consider the linear maps
 $$\rho(\a): \alg_\bullet \to G_\bullet,  u \mapsto u (a+\a) $$
 with $\a \in F_\bullet$.
Assume that for some $k \geq 0$, the following assumptions are satisfied:
 \begin{enumerate}[{\rm A)}]
\item for each $\a \in F_\bullet$, there is a tamed right quasi-inverse $j (\a) \in \Mt^k( G_\bullet,\alg_\bullet)$ of $\rho(\a)$ ;
\item the map   $j: F_\bullet \to \Mt^k( G_\bullet,\alg_\bullet),\ \a \mapsto j(\a)$ is tamed by  a sequence $(p_n)$.
\end{enumerate}
For any $b_0 \in M_0 $, we define the sequences 
$$ \left\{ \begin{matrix}
 \b_{n+1}&=&e^{-u_n} (a_n+\b_n)-a_{n+1} ;\\
  u_{n+1}&=&j(\sum_{i=0}^n\a_i)\left(\pi_G( \b_{n+1})\right) \end{matrix} \right.$$ 
 where $u_0=j(0)(\b_0) $ and
$$\left\{ \begin{matrix}   \a_{n}&=&\pi_F(u_n(a_n)-\b_n) ; \\
\g_n&=&\pi_G(u_n(a_n)-\b_n)\\  a_{n+1}&=& a_n+\a_n \end{matrix} .\right. $$ 
If $\b_n \in \Ht_d(E_n)$, then there exists a constant $B>0$   such that
 \begin{enumerate}[{\rm i)}]
 \item   $N^1_s(u_n)< B p_n^{-1} $ for any $s$ sufficiently small~;
 \item $  g(a+b)=r(a) (\mod F_\infty) $ where   $g$ is the limit of the sequence $(e^{u_n} e^{u_{n-1}} \dots  e^{u_1}  e^{u_0})_{n \in \NM} \subset \Lt(E_0,E_\infty)$.  
 \end{enumerate}
 \end{theorem}
  
Our task consists in showing that this theorem applies in our setting.  We assume the reader is acquainted with the notions
introduced in it (see~\cite{Abstract_KAM} for details). We first define the functional spaces to which the theorem applies.
\section{The spaces $C^\omega_n$, $L^{p,\omega}_n$, $C_\a^\omega(a)_\bullet$.}
 \subsection*{Definition of $C^\omega_n$}
  For any open subset $U \subset \CM^n$, the space of holomorphic functions in $U$ is endowed with the topology of uniform convergence on compact subset of $U$. If the open subset $U$ contains the origin, we get a restriction mapping
 $$r:\Ot_{\CM^n}(U) \to \Ot_{\CM^n,0}. $$
 Such mappings induce a direct limit topology on  $\Ot_{\CM^n,0}$~(see~\cite{Grothendieck_EVT} for details).
 
 We now construct a smaller directed system of vector spaces having $\Ot_{\CM^n,0}$ as limit. We denote by $(C^\omega_n)_s$,
  $s \in ]0,1/\sqrt{\pi}[$,  the vector space of continuous functions in the polydisc
  $$D_s:= \{ z \in \CM^n: \sup_{i=1,\dots,n }| z_i | \leq s \}$$
 which are holomorphic in its interior~
 $$(C^\omega_n)_s= \Ot(\mathring{D}_s) \cap C^0(D_s,\CM).$$
 The choice of $1/\sqrt{\pi}$ will simplify some estimates in the sequel.
 
  The $C^0$-norm
  $$| f |_s:=\sup_{z \in D_s} |f(z)| $$
endows  $(C^\omega_n)_s$ of a Banach space structure. The inclusion $D_s \subset D_t, t>s$ induces a directed system 
$$(C^\omega_n)_t \to (C^\omega_n)_s$$
which forms a Kolmogorov space. This directed system is of course standard~\cite{Douady_these,Nagumo}. On can define in a similar way the Kolmogorov spaces  $C^{k,\omega}$ by replacing continuous functions by $k$-differentiable ones and taking the $C^k$-norm.
 \subsection*{Definition of $L^{p,\omega}_n$}
We consider the Kolmogorov spaces $L^{p,\omega}_n$ :
$$(L^{p,\omega}_n)_s:= \G(\intr(D_s),\Ot_{\CM^n}) \cap L^p(D_s,\CM)$$
with $s \in ]0,1/\sqrt{\pi}[$.

For $p=2$, each of these spaces has a Hilbert space structure defined by the hermitian form
$$(f,g) \mapsto   \int_{D_s} f(z) \bar g(z) dV , $$
where $dV$ is the Euclidean volume on $\CM^n \approx \RM^{2n}$. In this Kolmogorov space, the  projection on a closed vector subspace is $0$-bounded. We now wish to extend this property to the Kolmogorov space $C^\omega_n$.

As any continuous function on a compact set is integrable, we get a canonical mapping~: 
$$I:C^\omega_n \to L^{2,\omega}_n,\ f \mapsto f .$$
This mapping can be distinguished from the identity only because the  source and target spaces are different Kolmogorov spaces. Such morphism
will be called {\em Identity morphism}.

As
$$\left( \int_{D_s} | f(z)|^2 dV \right)^{1/2} \leq \sup_{z \in D_s} \left| f(z) \right| \left(\sqrt{\pi}s\right)^{n}, $$
this identity morphism $I$ is $0$-bounded and
$$N_\tau^0(I)=\left(\sqrt{\pi}\tau\right)^{n}.  $$
In particular as $\tau \leq 1/\sqrt{\pi}$, we have $N_\tau^0(I)<1$ (hence the above choice of the interval in the definition of $C_n^\omega$).

Conversely, take a function $f \in (L^{2,\omega}_n)_t$ for any $s<t$ we have $f \in (C^\omega_n)_s$. Thus we also get another
identity morphism
$$I^{-1}:L^{2,\omega}_n \to  C^\omega_n$$
\begin{proposition}
\label{P::foncteur} The identity  morphism $I^{-1}$ is $1$-bounded of norm at most equal to one.
\end{proposition}
\begin{proof}
We denote by $|\cdot |_t$ the norm in the Hilbert space $(L^{2,\omega}_n)_t$.
Take   $f \in (L^{2,\omega}_n)_t$, the Taylor expansion at $w \in D_s$ with $s<t$ gives
$$f(z)=\sum_{j \geq 0} a_j (z-w)^j,\ a_j \in \CM. $$
Now let $\G_w$ be the polydisk centred at $w$ with radius $\s=|t-s|$. We have
$$\int_{\G_w} | f(z)|^2 dV=\sum_{j \geq 0} |a_j|^2 \s^{2j+2}  $$
and
$$\int_{\G_w} | f(z)|^2 dV \leq \int_{D_t} | f(z)|^2 dV=| f |_t^2  $$
This shows that
$$ |f(w) |=|a_0| \leq  \s^{-1}\left( \int_{\G_w} | f(z)|^2 dV \right)^{1/2} \leq  \s^{-1} | f |_{s+\s}$$
for any $w \in D_s$ and proves the proposition.
\end{proof}
\begin{corollary}
\label{C::projection} Let $F$ be a closed subspace of a Kolmogorov space $E$ compatible with the maps of the directed system:
$$\xymatrix{F_t \ar[r] \ar[d] & F_s  \ar[d]\\
E_t  \ar[r]& E_s
}$$
for any $t>s$.  There exists a one-bounded projection   $C^\omega_n \to F$.
\end{corollary}
\begin{proof} Denote by $\widetilde F$ the completion of $F$ in $L^{\omega,2}_n$ and consider the orthogonal projection
$$\pi:L^{\omega,2}_n \to \widetilde{F} .$$
For any $s<t$, we define a one-bounded projection from $C^\omega_n$ to $F$ using the commutative diagram:
$$\xymatrix{(C^\omega_n)_t \ar[r] \ar[d]_I & F_s  \\
(L^{\omega,2}_n)_t  \ar[r]^-\pi& \widetilde{F}_t \ar[u]_{I^{-1}_{|\widetilde{F}_t}}
}$$
\end{proof}
 \subsection*{Definition of $L^{\infty,\omega}_n$}
Let $(L^{\infty,\omega}_n)_s $ be the space of convergent power   series in $z_1,\dots,z_n$
such that the quantity
$$| f |_s^\infty=\sup_{i} | a_i |s^{|i|}, |i|:=i_1+i_2+\dots+i_n$$ 
is finite. This defines a Kolmogorov space that we denote by $L^{\infty,\omega}_n$. Recall that a series is called of {\em order $N$} if, in its Taylor expansion, all terms of degree less than $N$ vanish.

\begin{lemma}
\label{L::UV}
 If $f \in L^{\infty,\omega}_n $ is of order $N$ then
$$  | f |^\infty_s  \leq | f |^\infty_{s+\s} \left(\frac{s}{s+\s}\right)^{N} $$
\end{lemma}
\begin{proof}
Put $ f=\sum_{|i| \geq N } a_i z^i$, we have
$$ | a_i| s^{|i|} = | a_i| (s+\s)^{|i|}  \left(\frac{s}{s+\s}\right)^{|i|}  \leq | f |^\infty_{s+\s} \left(\frac{s}{s+\s}\right)^{N}$$
\end{proof}
This lemma relates the harmonic filtration to the usual filtration by formal power series. We now wish to have a similar result for $C^\omega_n$

For this, we consider the identity morphism 
$$I:C^\omega_n \to L^{\infty,\omega}_n.$$
This identity morphism factors through the identity morphism
$$ C^\omega_n \to L^{2,\omega}_n.$$
 It is therefore $0$-bounded of norm at most  equal to $1$.

\begin{proposition}
\label{L::Ck} The identity morphism $$I^{-1}: L^{\infty,\omega}_n \to C^\omega_n $$
is $n$-bounded with norm at most equal to $1$.  
\end{proposition}
\begin{proof}
We denote by $| \cdot |_t$ the norm in $(L^{\infty,\omega}_n)_t$.
Take $f \in  (C^\omega_n)_t$ and chose  $z \in D_s$ with $t>s$.
We have:
$$| f(z)| \leq \sum_{i} | a_i| s^{|i|} = \sum_{i} | a_i| t^{|i|} \left(\frac{s}{t}\right)^{|i|}  .$$
Using the estimate
$$  | a_i| t^{|i|} \leq | f |_t, $$
we get that
$$ |f(z) |\leq \left( \frac{t}{t-s} \right)^n | f |_t \leq  \frac{1}{(t-s)^n} | f |_t .$$
  \end{proof}
  Using Lemma~\ref{L::UV}, we get the
 \begin{corollary} 
 \label{C::foncteur}
 For any $f \in (C^\omega_n)_t^{(2^N)}$ and $s < t \leq 1$ we have:
  $$  | f |_s  \leq \frac{1}{(t-s)^n}| f |_t \left(\frac{s}{t}\right)^{2^N}. $$
  \end{corollary}

 \subsection*{Definition of $C_\a^\omega(a)_\bullet$.} 
We now fix a real positive decreasing sequence $a=(a_n)$ bounded by $1$. We
define the closed subsets
 $$\Ct(a)_n:=\{\a \in \CM^d: \s(\a)_n  \geq  a_n,\ \forall n \leq 2^n \},\ n \in \NM \cup \{ \infty \} $$
so that $\Ct(a)$ is now denoted by $\Ct(a)_\infty$.   Take $\a \in \Ct(a)$ and let $D_s(\a) \subset \CM^d$  be the closed polydisk of radius $s$ centred at $\a$.  Put $K_{n,s}=\Ct((1-s)a)_n \cap D_s(\a)$.
Note that for any $x \in  K_{n,s}$, the ball of radius $a_n/2^n (t-s)$ is contained inside $K_{n,t}$ for any $t>s$ (we will use this remark
 in the appendix).  
 
 The supremum norm induces a Banach space structure on the space $C_\a^\omega(a)_{n,s}$ of functions continuous on $K_{n,s}$ and holomorphic in its interior:
 $$C_\a^\omega(a)_{n,s}=\Ot(\mathring{K}_{n,s}) \cap C^0(K_{n,s}). $$
 
For each $n$, the inclusions $K_{n,s} \subset K_{n,t}, t>s$, $K_{n+1,s} \subset K_{n,s}$ induce a doubly directed system
$$\xymatrix{\cdots  \ar[r]& C_\a^\omega(a)_{n,t} \ar[r] \ar[d] &  C_\a^\omega(a)_{n+1,t}  \ar[d] \ar[r] & \cdots  \\
\cdots \ar[r]& C_\a^\omega(a)_{n,s}  \ar[r]   & C_\a^\omega(a)_{n+1,s}   \ar[r] & \cdots  \\  }$$
In this way, we defined the {\em Arnold space} $C_\a^\omega(a)_\bullet$. If $\a \in \CM^n$ is the origin and $a$ is the zero sequence then $$C_\a^\omega(a)_\bullet=C_n^\omega.$$
There are, of course, Arnold spaces $C_\a^{k,\omega}(a)$ obtained by replacing $C^0$-norms $|\cdot |^0_{n,s}$
by $C^k$-norms $|\cdot |^k_{n,s}$:
$$| f  |^k_{n,s}:=\max_{| j | \leq k} \frac{1}{|j|!} |\d^j f  |^0_{n,s}  $$
where $j=(j_1,\dots,j_n)$ and $|j|=j_1+j_2+\dots+j_n$.

To prove the main theorem,  we  apply the abstract KAM theorem to space of functions in the $4n$ variables $(t,\l,q,p)$:
$$E_\bullet:=C_\a^\omega(a)_\bullet \hat \otimes C^\omega_{3n} ,\ M_\bullet= E_\bullet^{(3)},\ F_\bullet:=
E_\bullet^{(3)}  \cap \left(I^2 \oplus (C^\omega_{n} \hat \otimes  C_\a^\omega(a)_\bullet  )\right) $$
and take $a=H$. The subspace $\alg_\bullet$ consists of sequences of Hamiltonian derivations. To prove the theorem,
it remains to construct a quasi-inverse $j$ to the infinitesimal action $\rho$.

 \section{Construction of the quasi-inverse}
 Let us denote by $R$ the Arnold space $C_\a^\omega(a)_\bullet  \hat \otimes C^\omega_{n}$ with coordinates functions $t_1,\dots,t_n$, $\l_1,\dots,\l_n$.
 
  We define a complement $G_\bullet$ to $F_\bullet$ as a sum of three $R$-modules
  $$G_\bullet= A_\bullet \oplus B_\bullet \oplus C_\bullet $$
with (we use multi-index notations):
\begin{enumerate}
\item  $A_\bullet=\bigoplus_{i=1}^n R (p_iq_i-\l_i) $ ;
\item $B_\bullet=\overline{\bigoplus_{I-J \neq 0} R p^I q^J}$ ;
\item   $C_\bullet:=\bigoplus_{i=1}^n (p_iq_i-\l_i) B_\bullet$.
\end{enumerate}
where $\overline{\ \cdot \ }$ denotes the closure.

According to Proposition~\ref{P::foncteur}, $F_\bullet, G_\bullet$  define a $1$-bounded splitting.  
Let us denote by $\star$ the Hadamard product for series, that is, the series obtained by taking the products coefficientwise:
$$\left(\sum_{n \geq 0} a_n z^n \right) \star \left( \sum_{n \geq 0} b_n z^n \right)= \sum_{n \geq 0} a_nb_n z^n $$
 \begin{lemma}
 \label{L::Hadamard} 
The Hadamard products
$$ \star h_\bullet:  R \hat \otimes  L^{2,\omega}_n \to  R \hat \otimes  L^{2,\omega}_n$$
with the functions
$$ h_k(\a,q,p):= \sum_{\| i-j\|=1}^{2^k} \frac{1}{ (\a,i-j)} q^ip^j,\ a_i \in \CM $$
define a $0$-tamed morphism whose norm is bounded from above by the sequence $a^{-1}$.
\end{lemma}
 \begin{proof}
Write 
$$f=\sum_{i \geq 0} f_{ij} q^ip^j,\ f_{ij} \in R. $$ 
We have
$$ f \star h_k(q,p) =  \sum_{\| i-j\|=1}^{2^k}  \frac{f_{ij}}{ (\a,i-j)}q^ip^j   ,$$
thus
$$  | f \star h_k |_{k,s} \leq \frac{1}{\s(\a)_k}\sqrt{   \sum_{\| i-j\|=1}^{2^k} |f_{ij}|_{k,s}^2 |q^ip^j |_s } \leq  \frac{1}{a_k}| f |_{k,s}.$$
This proves the lemma.
\end{proof}
In the decomposition $A_\bullet \oplus B_\bullet \oplus C_\bullet$, the operator $\rho_\bullet(f)$ admits the  lower triangular decomposition
$$\begin{pmatrix}   \Id & 0 & 0 \\ 0 &   \star g & 0 \\ 0 &  \{ -,f \} &  \star g  \\  \end{pmatrix} $$
with
$$g:=\sum_{i \neq j}(\a+t,(i-j)) q^i p^j.  $$

As the matrix is lower triangular, we can compute the inverse explicitly over formal power series. We truncate these series and define a right
quasi-inverse $j_\bullet(f)$ to $\rho(f)$ by putting
$$j_k(f):=\begin{pmatrix}   \Id & 0 & 0 \\ 0 & \mu_k& 0 \\ 0 & \mu_k  \{- ,  f \} \mu_k &   \mu_k  \end{pmatrix} $$
where $\mu_k=- \star h_k$.
Corollary \ref{C::foncteur} shows that it is indeed a quasi-inverse to $\rho(f)$.

 Lemma \ref{L::Hadamard} and Proposition \ref{P::foncteur} imply that the Hadamard product with $h_\bullet$ is $1$-tamed with norm  bounded from above by the sequence $a^{-1}$.  This shows that our quasi-inverse satisfies condition A) of the abstract KAM theorem.

Let us now check condition B).  The map
$$f \mapsto j_k(f) $$  involves first order derivatives of $f$.  Due to Cauchy inequalities, it is therefore $1$-tamed. This shows condition B) of the abstract KAM theorem and concludes the proof of the theorem.
\section*{Appendix :  Regularity in $C^{l,\omega}_\a(a)$. }
We generalise the regularity of KAM tori discovered  by P\"oschel~\cite{Poeschel_KAM}. Gevrey regularity can be treated in a similar way.

The identity morphism
$$I:C^{l,\omega}_\a(a) \to C^{0,\omega}_\a(a) $$
is $0$-bounded with norm at most one. The Cauchy inequalities give a partial converse~:
\begin{lemma}
\label{L::Cauchy} The identity morphism
$$I^{-1}:C^{0,\omega}_\a(a) \to C^{l,\omega}_\a(a) $$ is $l$-bounded and its norm is bounded by the sequence $ (2^{ln}a_n^{-l}) $
\end{lemma}
\begin{proof}
Take $s<t $ and $f \in (C^{l,\omega}_\a(a))_{n,t}$ . 
Put $K_{n,s}=\Ct((1-s)a)_n \cap D_s(\a)$.  For any $x \in  K_{n,t}$, the ball of radius $a_n/2^n (t-s)$ is contained inside $K_{n,s}$. Thus, the Cauchy inequalities imply that
$$  | I(f) |_{ s} \leq  \frac{2^{ln}}{(t-s)^l  a_n^l}   | f |_t .$$
This proves the lemma.

\end{proof} 
In particular, any $k$-bounded $u \in \Bt^k_\tau( C^{0,\omega})$  induces a $(k+l)$-bounded
morphism  $v  \in \Bt^{k+l}_\tau( C^{l,\omega}_\a(a))$ 
for which there is a commutative diagram
$$\xymatrix{ C^{l,\omega}_\a(a) \ar[r]^v \ar[d]^{I} & C^{l,\omega}_\a(a)  \ar[d]^{I}\\
C^{0,\omega}_\a(a) \ar[r]^u & C^{0,\omega}_\a(a) 
}$$
\begin{corollary}
\label{C::Cauchy}Let $a=(a_n)$ be a positive decreasing sequence and $l$ a positive real number. For any $k$-bounded $\tau $-morphism 
$u \in \Bt^k_\tau( C^{0,\omega}(a))$, 
 the norm of the induced morphism $v  \in \Bt^{k+l}_\tau( C^{l,\omega}_\a(a))$ satisfies the estimate
$$N^{k+l}_\tau(v_n) \leq \frac{2^{(k+l)(n+1)}}{a_n^{k+l}}N^{k}_\tau(u_n)  $$
for any $n \geq 0$.
\end{corollary} 
\begin{proof}
 Take $s<t \leq \tau $ and $f \in (C^{l,\omega}_\a(a))_{n,t}$ . 
We cut the interval $t-s$ into two equal pieces and write $v=I^{-1} uI$. The previous lemma shows that:
$$  | I^{-1} uI(f) |_{ s} \leq  \frac{2^{l(n+1)}}{(t-s)^l  a_n^l}   |uI( f) |_{s+\s} $$
with $\s=(t-s)/2$.  As $u$ is $k$-bounded, we have that
$$    | uI(f) |_{s+\s}  \leq  \frac{2^{k(n+1)}}{(t-s)^k  a_n^k}N^k_t(u_n)    | I(f) |_t$$
As $I$ is $0$-bounded with norm $1$, we have 
$$| I(f) |_t \leq | f|_t.$$
 This proves the corollary.
\end{proof}
 \begin{corollary} Let $a=(a_n)$ be a decreasing positive sequence and $u$ a $k$-bounded $\tau$ morphism of $C^{0,\omega}_\a(a)$, for some $\tau>0$. Assume that $(N^k_\tau(u_n))$ decreases faster than any power
 of $(2^{-n}a_n)$ :
 $$N^k_\tau(u_n)=o(2^{-jn}a_n^j), \forall j>0.$$ For any $l>0$, the norm of the morphism 
 $$v: C^{l,\omega}_\a(a) \to C^{l,\omega}_\a(a) $$  induced by $u$ has the same property:
 $$N^k_\tau(v_n)=o(2^{-jn}a_n^j),\ \forall j>0.$$
  \end{corollary}
This shows that the invariant lagrangian varieties obtained by applying the abstract KAM theorem have a $C^\infty$ dependence on parameters.

 
 \bibliographystyle{amsplain}
\bibliography{master}
 \end{document}